\documentclass[a4paper,12pt]{article}

\usepackage {amsfonts, amssymb, amscd,amsthm, amsmath, eucal}

\usepackage{amsbsy}
\usepackage[all]{xy}

\usepackage{color}

\theoremstyle{plain} {
  \newtheorem{thm}{Theorem}[section]
  \newtheorem{defn}[thm]{Definition}
  \newtheorem{cor}[thm]{Corollary}
  \newtheorem{lem}[thm]{Lemma}
  \newtheorem{prop}[thm]{Proposition}
  \theoremstyle{definition}
  \newtheorem{rem}[thm]{Remark}
  \theoremstyle{plain}
  \newtheorem{clm}[thm]{Claim}

  \newtheorem{sssection}[thm]{}
}

{

}
\renewcommand{\subsubsection}{\sssection\rm}

\newcommand{\can}{\text{\rm can}}

\newcommand{\id}{\text{\rm id}}
\newcommand{\pr}{\text{\rm pr}}

\newcommand{\inc}{\text{\rm inc}}

\newcommand{\const}{\text{\rm const}}
\newcommand{\Spec}{\text{\rm Spec}}

\newcommand{\Aff}{\mathbf {A}}
\newcommand{\Pro}{\mathbf {P}}

\newcommand \xra {\xrightarrow }

\newcommand \hra {\hookrightarrow }

\newcommand{\ttf}{{\text{f}}}

\newcommand\mydim{\text{\rm dim}}

\renewcommand \id{\operatorname{id}}

\renewcommand \phi\varphi

\newcommand{\et}{\text{\rm\'et}}

\newcommand{\ZZ}{\mathbb Z}

\begin{document}

\title
{
On Grothendieck-Serre conjecture
concerning principal $G$-bundles over regular semi-local domains containing a finite field: I
}

\author{Ivan Panin\footnote{The author acknowledges support of the
RNF-grant 14-11-00456.}
}


\maketitle

\begin{abstract}
In three preprints
\cite{Pan2},
\cite{Pan3}
and the present one
we prove
Grothendieck-Serre's conjecture concerning principal $G$-bundles over regular semi-local domains $R$
containing {\bf a finite field}
(here $G$ is a reductive group scheme).
The present preprint contains main geometric presentation theorems
which are necessary for that.
The preprint
\cite{Pan2}
contains reduction of the Grothendieck-Serre's conjecture to the case
of a simple simply-connected group scheme.
The preprint
\cite{Pan3}
contains a proof of
Grothendieck-Serre's conjecture for regular semi-local domains $R$
containing a finite field.
One of the main result of the present preprint is Theorem
\ref{MainHomotopy}.

The Grothendieck--Serre conjecture for the case of regular semi-local domains containing
{\bf an infinite field}
is proven in
\cite{FP}.
{\it Thus the conjecture holds for regular semi-local domains containing a field.}

We use results on Bertini theorems from
\cite{Poo1}
and
\cite{Poo2}
to get an appropriative elementary fibration
(Proposition \ref{ArtinsNeighbor}).
The present preprint is inspired by
\cite{PSV}.


\end{abstract}

\section{Introduction}
\label{Introduction}

Recall that an $R$-group scheme $G$ is called reductive
(respectively, semi-simple or simple), if it is affine and smooth as
an $R$-scheme and if, moreover, for each ring homomorphism
$s:R\to\Omega(s)$ to an algebraically closed field $\Omega(s)$, its
scalar extension $G_{\Omega(s)}$
is a connected reductive (respectively, semi-simple or simple)
algebraic group over $\Omega(s)$. The class of reductive group
schemes contains the class of semi-simple group schemes which in
turn contains the class of simple group schemes. This notion of a
simple $R$-group scheme coincides with the notion of a {simple
semi-simple $R$-group scheme from Demazure---Grothendieck
\cite[Exp.~XIX, Defn.~2.7 and Exp.~XXIV, 5.3]{SGA3}.} {\it
Throughout the paper $R$ denotes an integral domain and $G$ denotes
a semi-simple $R$-group scheme, unless explicitly stated otherwise.
All commutative rings that we consider are assumed to be
Noetherian.}
\par
A semi-simple $R$-group scheme $G$ is called {\it simply
connected\/} (respectively, {\it adjoint\/}), provided that for an
inclusion $s:R\hra\Omega(s)$ of $R$ into an algebraically closed
field $\Omega(s)$ the scalar extension $G_{\Omega(s)}$ is a simply
connected (respectively, adjoint) $\Omega(s)$-group scheme. This
definition coincides with the one from \cite[Exp.~XXII.
Defn.~4.3.3]{SGA3}.
\par
\par
A well-known conjecture due to J.-P.~Serre and A.~Grothendieck
\cite[Remarque, p.31]{Se}, \cite[Remarque 3, p.26-27]{Gr1}, and
\cite[Remarque 1.11.a]{Gr2} asserts that given a regular local ring
$R$ and its field of fractions $K$ and given a reductive group
scheme $G$ over $R$ the map
$$ H^1_{\text{\'et}}(R,G)\to H^1_{\text{\'et}}(K,G), $$
\noindent induced by the inclusion of $R$ into $K$, has trivial
kernel.

In three preprints \cite{Pan2}, \cite{Pan3} an the present one we prove this conjecture
for regular semi-local domains containing a finite field. For such a domain containing
an infinite field the conjecture is proved in
\cite{FP}.
{\it Thus the conjecture holds for regular semi-local domains containing a field.}

The preprint \cite{Pan3} contains a proof of the conjecture for regular
semi-local domains containing a finite field.
The general plan of the proof of the conjecture realized in
\cite{Pan3} is this. Using D.Popescu theorem the case of an arbitrary
regular local domain containing a finite field is reduced
to the case of regular local domains of the form as in Theorem
\ref{MainThmGeometric}. Next, Theorem
\ref{MainThmGeometric}
is used to reduce the case of an arbitrary reductive group scheme
to the case of semi-simple simply connected group schemes.
The latter case is easily reduced to the case of
simple simply connected group schemes.
Finally, using Theorem
\ref{MainHomotopy}
the case of simple simply connected group schemes
is proved in \cite{Pan3}. This latter case is proved
using all the technical tools developed in
\cite{FP}. Since we work over finite fields
results from
\cite{Poo1} and \cite{Poo2} are heavily used.

\begin{thm}
\label{MainHomotopy}
Let $k$ be a finite field. Let $\mathcal
O$ be the semi-local ring of finitely many {\bf closed points} on a
$k$-smooth irreducible affine $k$-variety $X$ and let $K$ be its
field of fractions. Let $G$ be a simple simply connected
group scheme over $\mathcal O$.
Let $G$ be a
{\bf semi-simple} simply connected group scheme over $\mathcal O$.
Let $\mathcal G$ be a principal $G$-bundle over $\mathcal O$
which is trivial over $K$. Then there
exists a principal $G$-bundle $\mathcal G_t$ over $\mathcal O[t]$
and a monic polynomial $h(t) \in \mathcal O[t]$ such
that
\par
(i) the $G$-bundle $\mathcal G_t$ is trivial over $\mathcal O[t]_h$,
\par
(ii) the evaluation of $\mathcal G_t$ at $t=0$ coincides
with the original $G$-bundle $\mathcal G$.

\end{thm}
Clearly, this Theorem looks similarly to the Theorem 1.2 from
\cite{PSV}.
However the proof of Theorem
\ref{MainHomotopy}
is much more involved
since the base field is finite.

\begin{thm}\label{MainThmGeometric}
Let $k$ be {\bf a finite field}. Assume that for any irreducible $k$-smooth affine variety $X$ and any
finite family of its {\bf closed points} $x_1, x_2,\dots, x_n$ and the semi-local $k$-algebra
$\mathcal O:= \mathcal O_{X,x_1, x_2,\dots, x_n}$
and all semi-simple simply connected reductive $\mathcal O$-group schemes $H$
the pointed set map
$$H^1_{\text{\rm\'et}}(\mathcal O, H) \to H^1_{\text{\rm\'et}}(k(X), H),$$
\noindent
induced by the inclusion of $\mathcal O$ into its fraction field $k(X)$, has trivial kernel.

Then for any regular semi-local domain $\mathcal O$ of the form
$\mathcal O_{X,x_1, x_2,\dots, x_n}$ above
and any reductive $\mathcal O$-group scheme $G$
the pointed set map
$$H^1_{\text{\rm\'et}}(\mathcal O, G) \to H^1_{\text{\rm\'et}}(K, G),$$
\noindent
induced by the inclusion of $\mathcal O$ into its fraction field $K$, has trivial kernel.
\end{thm}
Theorem
\ref{MainHomotopy}
is one of the main result of the present preprint.
Theorem
\ref{MainThmGeometric} is one of the main result of the preprint
\cite{Pan2}.

The preprint is organized as follows.
In Section 2 elementary fibrations are discussed.
In Section 3 the concept of a nice triple is recalled and Theorems 3.3 and 3.4 are formulated.
In Section 4 Theorem 3.4 is proved.
In Section 5 Theorem 3.3 is proved.
In Section 6 a basic nice triple is constructed.
In Section 7 Theorem
\ref{MainHomotopy}
is proved.

\medskip
The author thanks very much A.Suslin for his interest to the topic of the present preprint.
\section{Elementary fibrations}
\label{ElementaryFibrations}

In this Section we modify a result of M. Artin from~\cite{LNM305}
concerning existence of nice neighborhoods. The following notion is
a modification of the one introduced by Artin in~\cite[Exp. XI,
D\'ef. 3.1]{LNM305}.
\begin{defn}
\label{DefnElemFib} An elementary fibration is a morphism of
schemes $p:X\to S$ which can be included in a commutative diagram
\begin{equation}
\label{SquareDiagram}
    \xymatrix{
     X\ar[drr]_{p}\ar[rr]^{j}&&
\overline X\ar[d]_{\overline p}&&Y\ar[ll]_{i}\ar[lld]_{q} &\\
     && S  &\\    }
\end{equation}
of morphisms satisfying the following conditions:
\begin{itemize}
\item[{\rm(i)}] $j$ is an open immersion dense at each fibre of
$\overline p$, and $X=\overline X-Y$; \item[{\rm(ii)}] $\overline
p$ is smooth projective all of whose fibres are geometrically
irreducible of dimension one; \item[{\rm(iii)}] $q$ is finite
\'{e}tale all of whose fibres are non-empty.
\end{itemize}
\end{defn}

\begin{rem}
\label{ElementraryAndAlmostElementary}
Clearly, an elementary fibration is an almost elementary fibration
in the sense of \cite[Defn.2.1]{PSV}.
\end{rem}

Using repeatedly \cite[Thm.1.3]{Poo1} and
\cite[Thm.1.1]{Poo2}
and
modifying Artin's arguments
\cite[Exp. XI,Prop. 3.3]{LNM305},
one can prove the following result, which is a
slight extension of Artin's result \cite[Exp. XI,Prop. 3.3]{LNM305}.

\begin{prop}
\label{ArtinsNeighbor} Let $k$ be {\bf a finite field}, $X$ be a smooth
{\bf geometrically} irreducible {\bf affine} variety over $k$,
$x_1,x_2,\dots,x_n\in X$ be a family of {\bf closed} points. Then there exists a
Zariski open neighborhood $X^0$ of the family
$\{x_1,x_2,\dots,x_n\}$ and {\bf an elementary fibration} $p:X^0\to
S$, where $S$ is an open sub-scheme of the projective space
$\Pro^{\mydim X-1}$.
\par
If, moreover, $Z$ is a closed co-dimension one subvariety in $X$,
then one can choose $X^0$ and $p$ in such a way that $p|_{Z\bigcap
X^0}:Z\bigcap X^0\to S$ is finite surjective.
\end{prop}

The following result is proved in
\cite[Prop.2.4]{PSV}.

\begin{prop}
\label{CartesianDiagram} Let $p: X \to S$ be an elementary
fibration. If $S$ is a regular semi-local irreducible scheme, then
there exists a commutative diagram of $S$-schemes
\begin{equation}
\label{RactangelDiagram}
    \xymatrix{
X\ar[rr]^{j}\ar[d]_{\pi}&&\overline X\ar[d]^{\overline \pi}&&
Y\ar[ll]_{i}\ar[d]_{}&\\
\Aff^1\times S\ar[rr]^{\text{\rm in}}&&\Pro^1\times S&&
\ar[ll]_{i}\{\infty\}\times S &\\  }
\end{equation}
\noindent such that the left hand side square is Cartesian. Here $j$
and $i$ are the same as in Definition \ref{DefnElemFib}, while
$\pr_S \circ\pi=p$, where $\pr_S$ is the projection $\Aff^1\times
S\to S$.

In particular, $\pi:X\to\Aff^1\times S$ is a finite surjective
morphism of $S$-schemes, where $X$ and $\Aff^1\times S$ are regarded
as $S$-schemes via the morphism $p$ and the projection $\pr_S$,
respectively.
\end{prop}

\section{Nice triples}
\label{NiceTriples}

In the present section we introduce and study certain collections
of geometric data and their morphisms. The concept of a {\it nice
triple\/} is very similar to that of a {\it standard triple\/}
introduced by Voevodsky \cite[Defn.~4.1]{Vo}, and was in fact
inspired by the latter notion. Let $k$ be {\bf a finite field}, let $X/k$
be a smooth geometrically irreducible {\bf affine} variety, and let
$x_1,x_2,\dots,x_n\in X$ be {\bf a family of closed points}. Further, let
$\mathcal O=\mathcal O_{X,\{x_1,x_2,\dots,x_n\}}$ be the
corresponding geometric semi-local ring.
\begin{defn}
\label{DefnNiceTriple} Let $U:=\text{\Spec}(\mathcal
O_{X,\{x_1,x_2,\dots,x_n\}})$.
A \emph{nice triple} over $U$ consists of the following data:
\begin{itemize}
\item[\rm{(i)}] a smooth morphism $q_U:\mathcal X\to U$, where $\mathcal X$ is an irreducible scheme,
\item[\rm{(ii)}] an element $f\in\Gamma(\mathcal X,\mathcal
O_{\mathcal X})$,
\item[\rm{(iii)}] a section $\Delta$ of the morphism $q_U$,
\end{itemize}
subject to the following conditions:
\begin{itemize}
\item[\rm{(a)}]
each irreducible component of each fibre of the morphism $q_U$
has dimension one,
\item[\rm{(b)}]
the module
$\Gamma(\mathcal X,\mathcal O_{\mathcal X})/f\cdot\Gamma(\mathcal X,\mathcal O_{\mathcal X})$
is finite as
a $\Gamma(U,\mathcal O_{U})=\mathcal O$-module,
\item[\rm{(c)}]
there exists a finite surjective $U$-morphism
$\Pi:\mathcal X\to\Aff^1\times U$,
\item[\rm{(d)}]
$\Delta^*(f)\neq 0\in\Gamma(U,\mathcal O_{U})$.
\end{itemize}
\end{defn}

\begin{defn}
\label{DefnMorphismNiceTriple} A \emph{morphism}
between two nice
triples
over $U$
$$(q^{\prime}: \mathcal X^{\prime} \to U,f^{\prime},\Delta^{\prime})\to(q: \mathcal X \to U,f,\Delta)$$
is an \'{e}tale morphism of $U$-schemes $\theta:\mathcal
X^{\prime}\to\mathcal X$ such that
\begin{itemize}
\item[\rm{(1)}] $q^{\prime}_U=q_U\circ\theta$, \item[\rm{(2)}]
$f^{\prime}=\theta^{*}(f)\cdot h^{\prime}$ for an element
$h^{\prime}\in\Gamma(\mathcal X^{\prime},\mathcal O_{\mathcal
X^{\prime}})$,
\item[\rm{(3)}] $\Delta=\theta\circ\Delta^{\prime}$.
\end{itemize}
\end{defn}
Two observations are in order here.
\par\smallskip
$\bullet$ Item (2) implies in particular that $\Gamma(\mathcal
X^{\prime},\mathcal O_{\mathcal
X^{\prime}})/\theta^*(f)\cdot\Gamma(\mathcal X^{\prime},\mathcal
O_{\mathcal X^{\prime}})$ is a finite
$\mathcal O$-module.
\par\smallskip
$\bullet$ It should be emphasized that no conditions are imposed
on the interrelation of $\Pi^{\prime}$ and $\Pi$.
\par\smallskip

Let us state two crucial results which will be used in our main
construction. Their proofs are given in Sections
\ref{ProofThEquatingGroups}
and
\ref{SectElemNisnevichSquare}
respectively. If
$U$ as in Definition
\ref{DefnNiceTriple}
the for any $U$-scheme $V$ and any closed point $u \in U$
set $V_u=u\times_U V$.
For a finite set $A$ denote $\sharp A$ the cardinality of $A$.

\begin{thm}
\label{ThEquatingGroups} Let $U$ be as in Definition
\ref{DefnNiceTriple}. Let $(\mathcal X,f,\Delta)$ be a nice triple
over $U$. Let $G_{\mathcal X}$ be a
semi-simple
$\mathcal X$-group scheme, and let $G_U:=\Delta^*(G_{\mathcal
X})$. Finally, let $G_{\const}$ be the pull-back of $G_U$ to
$\mathcal X$. Then there exists a morphism
$\theta:(\mathcal X^{\prime},f^{\prime},\Delta^{\prime})\to(\mathcal X,f,\Delta)$
of nice triples {\bf over} $U$
satisfying the following conditions
\begin{itemize}
\item[\rm{(1)}]
there is an isomorphism
$ \Phi: \theta^*(G_{\const}) \to \theta^*(G_{\mathcal X}) $
of $\mathcal X^{\prime}$-group schemes such that
$(\Delta^{\prime})^*(\Phi)=\id_{G_U}$,
\item[\rm{(2)}]
for the closed sub-scheme $\mathcal Z^{\prime}$ of $\mathcal X^{\prime}$ defined by $\{f^{\prime}=0\}$
and any closed point $u \in U$ the point $\Delta^{\prime}(u) \in \mathcal Z^{\prime}_u$ is the only
$k(u)$-rational point of $\mathcal Z^{\prime}_u$,
\item[\rm{(3)}]
for any closed point $u \in U$ and any integer $r \geq 1$ and for $\mathcal Z^{\prime}$ as in \rm{(2)} one has
$$\sharp\{z \in \mathcal Z^{\prime}_u| \text{deg}[z:u]=r \} \leq \ \sharp\{x \in \Aff^1_u| \text{deg}[z:u]=r \}$$
\end{itemize}
\end{thm}

\begin{thm}
\label{ElementaryNisSquareNew} Let $U$ be as in Definition
\ref{DefnNiceTriple}.
Let $(\mathcal X^{\prime},f^{\prime},\Delta^{\prime})$ be a nice triple
over $U$, such that $f^{\prime}$ {\bf vanishes at every closed point of} $\Delta^{\prime}(U)$.
Let $\mathcal Z^{\prime}$ be the closed sub-scheme of $\mathcal X^{\prime}$ defined by $\{f^{\prime}=0\}$.
Assume that $\mathcal Z^{\prime}$ satisfies the conditions {\rm (2)} and {\rm (3)} from
Theorem
\ref{ThEquatingGroups}.
Then there exist a distinguished finite surjective morphism
$$ \sigma:\mathcal X^{\prime} \to\Aff^1\times U $$
\noindent
of $U$-schemes,
a monic polinomial
$h \in Ker[\mathcal O[t] \xrightarrow{\bar {} \circ \sigma^*}
\Gamma(\mathcal X^{\prime}, \mathcal O_{\mathcal X^{\prime}})/(f^{\prime})]$
and an element $g \in \Gamma(\mathcal X^{\prime},\mathcal O_{\mathcal X^{\prime}} )$
which enjoys the following properties:
\begin{itemize}
\item[\rm{(a)}]
the morphism $\sigma_g= \sigma|_{\mathcal X^{\prime}_g}: \mathcal X^{\prime}_g \to \Aff^1\times U $
is \'{e}tale,
\item[\rm{(b)}]
data
$ (\mathcal O[t],\sigma^*_g: \mathcal O[t] \to \Gamma(\mathcal X^{\prime},\mathcal O_{\mathcal X^{\prime}} )_g, h ) $
satisfies the hypotheses of
\cite[Prop.2.6]{C-TO},
i.e.
$\Gamma(\mathcal X^{\prime},\mathcal O_{\mathcal X^{\prime}} )_g$
is a finitely generated as the
$\mathcal O[t]$-algebra, the element $(\sigma_g)^*(h)$
is not a zero-divisor in
$\Gamma(\mathcal X^{\prime},\mathcal O_{\mathcal X^{\prime}} )_g$
and
$\mathcal O[t]/(h)=\Gamma(\mathcal X^{\prime},\mathcal O_{\mathcal X^{\prime}} )_g/h\Gamma(\mathcal X^{\prime},\mathcal O_{\mathcal X^{\prime}} )_g$ \ ,
\item[\rm{(c)}]
$(\Delta^{\prime}(U) \cup \mathcal Z^{\prime}) \subset \mathcal X^{\prime}_g$ \ ,
\item[\rm{(d)}]
$\mathcal X^{\prime}_{gh} \subseteq \mathcal X^{\prime}_{gf^{\prime}}$ \ .
\end{itemize}
\end{thm}

\begin{rem}
\label{ElementaryNisSquareRem}
The item \rm{(b)} of this theorem shows that the cartesian square
\begin{equation}
\label{SquareDiagram2}
    \xymatrix{
     \mathcal X^{\prime}_{gh}  \ar[rr]^{\inc} \ar[d]_{\sigma_{gh}} &&
     \mathcal X^{\prime}_g \ar[d]^{\sigma_g}  &\\
     (\Aff^1 \times U)_{h} \ar[rr]^{\inc} && \Aff^1 \times U &\\
    }
\end{equation}
can be used to glue principal $G$-bundles.
The items
\rm{(a)} and \rm{(b)} show that the square
(\ref{SquareDiagram2})
is an elementary {\bf distinguished} square in the category of smooth $U$-schemes in the sense of
\cite[Defn.3.1.3]{MV}.
The item \rm{(d)} guaranties that a principal $G$-bundle on $\mathcal X^{\prime}$,
which is trivial being restricted to
$\mathcal X^{\prime}_{f^{\prime}}$
is trivial being restricted to
$\mathcal X^{\prime}_{gh}$.
\end{rem}
\section{Proof of Theorem~\ref{ElementaryNisSquareNew}}
\label{SectElemNisnevichSquare}
The nearest aim is to
prove Theorem~\ref{ElementaryNisSquareNew}.
The following theorem is a step to do that.

\begin{thm}
\label{ElementaryNisSquare} Let $U$ be as in Definition
\ref{DefnNiceTriple}. Let $(q^{\prime}_U: \mathcal X^{\prime} \to U,f^{\prime},\Delta^{\prime})$ be a nice triple
over $U$, such that $f^{\prime}$ {\bf vanishes at every closed point of} $\Delta^{\prime}(U)$.
Let $\mathcal Z^{\prime}$ be the closed sub-scheme of $\mathcal X^{\prime}$ defined by $\{f^{\prime}=0\}$.
Assume that $\mathcal Z^{\prime}$ satisfies the conditions {\rm (2)} and {\rm (3)} from
Theorem
\ref{ThEquatingGroups}.
Then there exists a distinguished finite surjective morphism
$$ \sigma:\mathcal X^{\prime} \to\Aff^1\times U $$
\noindent
of $U$-schemes which enjoys the following properties:
\begin{itemize}
\item[\rm{(a)}]
the morphism
$\sigma|_{\mathcal Z^{\prime}}: \mathcal Z^{\prime} \to \Aff^1\times U$
is a closed embedding;
\item[\rm{(b)}] $\sigma$ is \'{e}tale in a neighborhood of
$\mathcal Z^{\prime}\cup\Delta^{\prime}(U)$;
\item[\rm{(c)}]
$\sigma^{-1}(\sigma(\mathcal Z^{\prime}))=\mathcal Z^{\prime}\coprod \mathcal Z^{\prime\prime}$
scheme theoretically and
$\mathcal Z^{\prime\prime} \cap \Delta^{\prime}(U)=\emptyset$;
\item[\rm{(d)}]
$\sigma^{-1}(\{0\} \times U)=\Delta^{\prime}(U)\coprod \mathcal D$ scheme theoretically and $\mathcal D \cap \mathcal Z^{\prime}=\emptyset$;
\item[\rm{(e)}]
for $\mathcal D_1:=\sigma^{-1}(\{1\} \times U)$ one has
$\mathcal D_1 \cap \mathcal Z^{\prime}=\emptyset$.
\item[\rm{(f)}]
there is a monic polinomial
$h \in \mathcal O[t]$
such that
$(h)=Ker[\mathcal O[t] \xrightarrow{\bar {} \circ \sigma^*}
\Gamma(\mathcal X^{\prime}, \mathcal O_{\mathcal X^{\prime}})/(f^{\prime})]$.
\end{itemize}
\end{thm}

\begin{proof}[Sketch of the proof of Theorem \ref{ElementaryNisSquare}]
For any closed point $u \in U$ and any $U$-scheme $V$ let $V_u=u\times_U V$ be
the fibre of the scheme $V$ over the point $u$.

Step (i). For any closed point $u \in U$ and any point $z \in \mathcal Z^{\prime}_u$ there is a closed embedding
$z^{(2)} \hra \Aff^1_u$, where
$z^{(2)}:=Spec(\Gamma(\mathcal X^{\prime}_u, \mathcal O_{\mathcal X^{\prime}_u})/\mathfrak m^2_z)$
for the maximal ideal
$\mathfrak m_z \subset \Gamma(\mathcal X^{\prime}_u, \mathcal O_{\mathcal X^{\prime}_u})$
of the point
$z$ regarded as a point of $\mathcal X^{\prime}$. This holds, since the $k(u)$-scheme
$\mathcal X^{\prime}_u$
is equi-dimensional of dimension one, affine and $k(u)$-smooth.

Step (ii). For any closed point $u \in U$ there is a closed embedding
$i_u: \coprod_{z \in \mathcal Z^{\prime}_u} z^{(2)} \hra \Aff^1_u$
of the $k(u)$-schemes. To see this apply Step (i) and use that
$\mathcal Z^{\prime}$ satisfies the condition {\rm (3)} from
Theorem
\ref{ThEquatingGroups}.

Step(iii) is to introduce some notation.
Since $(\mathcal X^{\prime},f^{\prime},\Delta^{\prime})$ is a nice triple
over $U$ there is a finite surjective morphism
$\mathcal X^{\prime} \xrightarrow{\Pi} \Aff^1\times U$
of the $U$-schemes.
Take the composite
$\mathcal X^{\prime} \xrightarrow{\Pi} \Aff^1\times U \hra \Pro^1\times U$
morphism and denote by
$\bar {\mathcal X^{\prime}}$ the normalization of $\Pro^1\times U$
in the fraction field
$k(\mathcal X^{\prime})$
of the ring
$\Gamma(\mathcal X^{\prime}, \mathcal O_{\mathcal X^{\prime}})$.
The normalization of
$\Aff^1 \times U$ in $k(\mathcal X^{\prime})$
coincides with
the scheme
$\mathcal X^{\prime}$,
since
$\mathcal X^{\prime}$ is a regular scheme.
So, we have a Cartesian diagram of $U$-schemes
\begin{equation}
\label{NormalizationSquare}
    \xymatrix{
     \mathcal X^{\prime}  \ar[rr]^{\inc} \ar[d]_{\Pi} &&
     \bar {\mathcal X^{\prime}} \ar[d]^{\bar \Pi}  &\\
     \Aff^1 \times U \ar[rr]^{\inc} && \Pro^1 \times U, &\\
    }
\end{equation}
in which the horizontal arrows are open embedding.

Let
$\mathcal X^{\prime}_{\infty}$
be the Cartie-divisor
$(\bar \Pi)^{-1}(\infty \times U)$
in
$\bar {\mathcal X^{\prime}}$.
Let
$\mathcal L:=\mathcal O_{\bar {\mathcal X^{\prime}}}(\mathcal X^{\prime}_{\infty})$
be the corresponding invertible sheaf and let
$s_0 \in \Gamma(\bar {\mathcal X^{\prime}}, \mathcal L)$
be its section vanishing exactly on
$\mathcal X^{\prime}_{\infty}$.
One has a Cartesian square of $U$-schemes
\begin{equation}
\label{SquareForClosedFibre}
    \xymatrix{
     \mathcal X^{\prime}_{{\infty},u}  \ar[rr]^{j_{\infty}} \ar[d]_{in_u} &&
     \mathcal X^{\prime}_{\infty} \ar[d]^{in}  &\\
     \bar {\mathcal X^{\prime}_u} \ar[rr]^{j} && \bar {\mathcal X^{\prime}}, &\\
    }
\end{equation}
which shows that the closed embedding
$in_u: \mathcal X^{\prime}_{{\infty},u} \hra \bar {\mathcal X^{\prime}_u}$
is a Cartie-divisor on
$\bar {\mathcal X^{\prime}_u}$.
Set
$\mathcal L_u=j^*(\mathcal L)$
and
$s_{0,u}=s_0|_{\bar {\mathcal X^{\prime}_u}} \in \Gamma(\bar {\mathcal X^{\prime}_u},\mathcal L_u)$.

Step (iv). There exists an integer $n>0$ and a section
$s_{1,u} \in \Gamma(\bar {\mathcal X^{\prime}_u},\mathcal L^{\otimes n}_u)$
which has no zeros on
$\mathcal X^{\prime}_{{\infty},u}$
and such that the morphism
$$[s^n_{0,u}: s_{1,u}]: \bar {\mathcal X^{\prime}_u} \to \Pro^1_u$$
has the following two properties \\
(a) the morphism $\sigma_u= s_{1,u}/s^n_{0,u}: \mathcal X^{\prime}_u \to \Aff^1_u$ is finite surjective,\\
(b) $\sigma_u|_{\coprod_{z \in \mathcal Z^{\prime}_u} z^{(2)}}=i_u: \coprod_{z \in \mathcal Z^{\prime}_u} z^{(2)} \hra \Aff^1_u$,
where $i_u$ is
from the step (ii); in particular,
$\sigma_u$ is \'{e}tale at every point
$z \in \mathcal Z^{\prime}_u$.

Step (v). There exists a section
$s_1 \in \Gamma(\bar {\mathcal X^{\prime}}, \mathcal L^{\otimes n})$
such that for any closed point $u \in U$ one has
$s_1|_{\mathcal X^{\prime}_u}=s_{1,u}$.

Step (vi). If $s_1$ is such as in the step (v), then the morphism
$$\sigma=(s_1/s^n_0, pr_U): \mathcal X^{\prime} \to \Aff^1\times U$$
is finite and surjective.

{\bf We are ready now to check step by step all the statements of the Theorem.}

{\it The assertion (b)}. Since the schemes $\mathcal X^{\prime}$ and $\Aff^1\times U$
are regular and the morphism $\sigma$ is finite and surjective, the morphism
$\sigma$ is flat by a theorem of Grothendieck.

So, to check that $\sigma$ is \'{e}tale at a closed point
$z \in \mathcal Z^{\prime}$
it suffices to check that for the point
$u=q^{\prime}_U(z)$
the morphism
$\sigma_u: \mathcal X^{\prime}_u \to \Aff^1_u$
is \'{e}tale at the point $z$. The latter does hold by
the step (iv), item (b).
Whence $\sigma$ is \'{e}tale at all the closed points of
$\mathcal Z^{\prime}$. By the hypotheses of the Theorem
the set of closed points of
$\Delta^{\prime}(U)$
is contained in the set of the closed points of
$\mathcal Z^{\prime}$.
Whence $\sigma$ is \'{e}tale also at all the closed points of
$\Delta^{\prime}(U)$. The schemes
$\Delta^{\prime}(U)$
and
$\mathcal Z^{\prime}$
are both semi-local.
Thus,
$\sigma$ is \'{e}tale in a neighborhood of
$\mathcal Z^{\prime} \cup \Delta^{\prime}(U)$.

{\it The assertion (a)}.
For any closed point $u \in U$ and
{\bf any point $z \in \mathcal Z^{\prime}_u$ }
the $k(u)$-algebra homomorphism
$\sigma^*_u: k(u)[t] \to k(u)[\mathcal X^{\prime}_u]$
is \'{e}tale at the maximal ideal
$\mathfrak m_z$ of the
$k(u)$-algebra
$k(u)[\mathcal X^{\prime}_u]$
and the composite map
$k(u)[t] \xrightarrow{\sigma^*_u} k(u)[\mathcal X^{\prime}_u] \to k(u)[\mathcal X^{\prime}_u]/\mathfrak m_z$
is an epimorphism. Thus, for any integer $r>0$ the $k(u)$-algebra homomorphism
$k(u)[t] \to k(u)[\mathcal X^{\prime}_u]/\mathfrak m^r_z$
is an epimorphism. The local ring
$\mathcal O_{\mathcal Z^{\prime}_u,z}$
of the scheme
$\mathcal Z^{\prime}_u$
at the point $z$
is of the form
$k(u)[\mathcal X^{\prime}_u]/\mathfrak m^s_z$
for an integer $s$. Thus, the $k(u)$-algebra homomorphism
$$k(u)[t] \xrightarrow{\sigma^*_u} k(u)[\mathcal X^{\prime}_u] \to \mathcal O_{\mathcal Z^{\prime}_u,z}$$
is surjective. Since
$\sigma_u|_{\coprod_{z \in \mathcal Z^{\prime}_u} z^{(2)}}=i_u$
and $i_u$ is a closed embedding one concludes that the
$k(u)$-algebra homomorphism
$$k(u)[t] \to \prod_{z/u}O_{\mathcal Z^{\prime}_u,z}=\Gamma(\mathcal Z^{\prime}_u, \mathcal O_{\mathcal Z^{\prime}_u})$$
is surjective.
Let
${\bf u}=\coprod Spec(k(u))$ regarded as the closed sub-scheme of $U$,
where $u$ runs over all closed points of $U$.
Then, for the scheme
$\mathcal Z^{\prime}_{{\bf u}}={\bf u} \times_U \mathcal Z^{\prime}$
the $k[{\bf u}]$-algebra homomorphism
\begin{equation}
\label{Finiteness}
k[{\bf u}][t] \to \Gamma(\mathcal Z^{\prime}_{{\bf u}}, \mathcal O_{\mathcal Z^{\prime}_{{\bf u}}})
\end{equation}
is surjective.

Since $(\mathcal X^{\prime},f^{\prime},\Delta^{\prime})$ is a nice triple over $U$,
the $\mathcal O$-module
$\Gamma(\mathcal Z^{\prime}, \mathcal O_{\mathcal Z^{\prime}})$
is finite.
Thus, the
$k[{\bf u}]$-module
$\Gamma(\mathcal Z^{\prime}_{{\bf u}}, \mathcal O_{\mathcal Z^{\prime}_{{\bf u}}})$
is finite.
Now the surjectivity of the $k[{\bf u}]$-algebra homomorphism
(\ref{Finiteness})
and the Nakayama lemma show that the $\mathcal O$-algebra homomorphism
$\mathcal O[t] \to \Gamma(\mathcal Z^{\prime}, \mathcal O_{\mathcal Z^{\prime}})$
is surjective. Thus,
$\sigma|_{\mathcal Z^{\prime}}: \mathcal Z^{\prime} \to \Aff^1\times U$
is a closed embedding.

{\it The assertion (e)}.
The morphism
$\Delta^{\prime}$
is a section of the structure morphism
$q^{\prime}_U: \mathcal X^{\prime} \to U$
and the morphism
$\sigma$
is a morphism of the $U$-schemes.
Hence
the composite morphism
$t_0:= \sigma \circ \Delta^{\prime}$
is a section of the projection
$pr_U: \Aff^1\times U \to U$.
This section is defined by an element
$a \in \mathcal O$.
There is another section
$t_1$ of the projection
$pr_U$
defined by
the element
$1-a \in \mathcal O$.
Making an affine change of coordinates on
$\Aff^1_U$ we may and will assume that
$t_0(U)=\{0\} \times U$
and
$t_1(U)=\{1\} \times U$.

Since $\mathcal D_1$ and $\mathcal Z^{\prime}$ are semi-local, to prove the assertion (e)
it suffices to check that
$\mathcal D_1$ and $\mathcal Z^{\prime}$
have no common closed points.
Let $z \in \mathcal D_1 \cap \mathcal Z^{\prime}$ be a common closed point.
Then
$\sigma(z) \in \{1\} \times U$.
Let $u=q^{\prime}_U(z)$. We already know that
$\sigma|_{\mathcal Z^{\prime}}$
is a closed embedding. Thus
$deg[z:u]=deg[\sigma(z):u]=1$.
The $U$-scheme $\mathcal Z^{\prime}$ satisfies the conditions {\rm (2)}
of Theorem
\ref{ThEquatingGroups}.
Thus, $z=\Delta^{\prime}(u)$.
In this case
$\sigma(z) \in \{0\} \times U$.
But
$\sigma(z) \in \{1\} \times U$.
This is a contradiction.
Whence
$\mathcal D_1 \cap \mathcal Z^{\prime}=\emptyset$.

{\it The assertion (c)}.
The finite morphism
$\sigma$
is \'{e}tale in a neighborhood of
$\mathcal Z^{\prime}$
by the item (b) of the Theorem.
By the item
(a) of the Theorem
$\sigma|_{\mathcal Z^{\prime}}$
is a closed embedding.
Thus, the morphism
$\sigma^{-1}(\sigma(\mathcal Z^{\prime})) \to \sigma(\mathcal Z^{\prime})$
of affine schemes is finite and
there is an affine open sub-scheme
$V$ of the scheme
$\sigma^{-1}(\sigma(\mathcal Z^{\prime}))$
such that the morphism
$V \to \sigma(\mathcal Z^{\prime})$
is \'{e}tale.
Since
$\sigma|_{\mathcal Z^{\prime}}$
is a closed embedding there is a unique section
$s$ of the morphism
$\sigma^{-1}(\sigma(\mathcal Z^{\prime})) \to \sigma(\mathcal Z^{\prime})$
with the image
$\mathcal Z^{\prime}$
and this image is contained in $V$.
By \cite[Lemma 5.3]{OP1} the scheme
$\sigma^{-1}(\sigma(\mathcal Z^{\prime}))$
has the form
$\sigma^{-1}(\sigma(\mathcal Z^{\prime}))=\mathcal Z^{\prime} \coprod \mathcal Z^{\prime\prime}$.\\
By a similar reasoning the scheme
$\sigma^{-1}(\{0\} \times U)$
has the form
$\Delta^{\prime}(U) \coprod \mathcal D$.
All the closed points of
$\Delta^{\prime}(U)$
are closed points of
$\mathcal Z^{\prime}$
and
$\mathcal Z^{\prime} \cap \mathcal Z^{\prime\prime}=\emptyset$.
Thus,
$\Delta^{\prime}(U) \cap \mathcal Z^{\prime\prime}=\emptyset$.

{\it The assertion (d)}.
It remains to show that
$\mathcal D \cap \mathcal Z^{\prime}=\emptyset$.
It suffices to check that
$\mathcal D$ and $\mathcal Z^{\prime}$
have no common closed points.
Let $z \in \mathcal D \cap \mathcal Z^{\prime}$
be a common closed point.
Then
$\sigma(z) \in \{0\} \times U$.
Let $u=q^{\prime}_U(z)$. We already know that
$\sigma|_{\mathcal Z^{\prime}}$
is a closed embedding. Thus
$deg[z:u]=deg[\sigma(z):u]=1$.
The $U$-scheme $\mathcal Z^{\prime}$ satisfies the conditions {\rm (2)}
of Theorem
\ref{ThEquatingGroups}.
Thus, $z=\Delta^{\prime}(u) \in \Delta^{\prime}(U)$.
So,
$z \in \Delta^{\prime}(U) \cap \mathcal D$.
But as we already know
$\Delta^{\prime}(U) \cap \mathcal D=\emptyset$.
This contradiction shows that
$\mathcal D$ and $\mathcal Z^{\prime}$
have no common closed points.
Thus,
$\mathcal D \cap \mathcal Z^{\prime}=\emptyset$.

{\it The assertion (f)}.
Recall that
$\mathcal X^{\prime}$ is affine irreducible and regular.
So, the principal ideal
$(f^{\prime})$ has the form
$\mathfrak p^{r_1}_1\mathfrak p^{r_2}_2\cdots\mathfrak p^{r_n}_n$,
where
$\mathfrak p_i$'s
are hight one prime ideals in
$\Gamma(\mathcal X^{\prime}, \mathcal O_{\mathcal X^{\prime}})$.
Let
$\mathcal Z^{\prime}_i$
be the closed subscheme in
$\mathcal X^{\prime}$
defined by the ideal
$\mathfrak p_i$.
Let
$\mathfrak q_i=\mathcal O[t] \cap \mathfrak p_i$.
The morphism
$\sigma|_{\mathcal Z^{\prime}}: \mathcal Z^{\prime} \to \Aff^1\times U$
is a closed embedding by the item (a) of Theorem
\ref{ElementaryNisSquare}.
This yields that
$\sigma|_{\mathcal Z^{\prime}_i}: \mathcal Z^{\prime}_i \to \Aff^1\times U$
is a closed embedding too. Thus
$\mathfrak p_i$
is a hight one prime ideal in
$\mathcal O[t]$.
So, it is a principal prime ideal.
Since
$\mathcal Z^{\prime}$
is finite over
$U$ the scheme
$\mathcal Z^{\prime}_i$
is finite over $U$ too.
Hence the principal prime ideal
$\mathfrak p_i$
is of the form
$(h_i)$ for a unique monic polinomial
$h_i \in \mathcal O[t]$.

Set $h=h^{r_1}_1h^{r_2}_2 \dots h^{r_n}_n$.
Clearly,
$h \in Ker[\mathcal O[t] \xrightarrow{\bar {} \circ \sigma^*}
\Gamma(\mathcal X^{\prime}, \mathcal O_{\mathcal X^{\prime}})/(f^{\prime})]$.
Since the map
$\mathcal O[t] \xrightarrow{\bar {} \circ \sigma^*} \Gamma(\mathcal X^{\prime}, \mathcal O_{\mathcal X^{\prime}})/(f^{\prime})$
is surjective, to prove the assertion (f) it suffices to check that the surjective $\mathcal O$-module homomorphism
$$\bar {} \circ \sigma^*: \mathcal O[t]/(h) \to \Gamma(\mathcal X^{\prime}, \mathcal O_{\mathcal X^{\prime}})/(f^{\prime})$$
is an isomorphism. Both sides are finitely generated projective
$\mathcal O$-modules. It remains to check that both sides have the same rank as the
$\mathcal O$-modules. For that it suffices to know that
$\mathcal O[t]/(h_i)$ and $\Gamma(\mathcal X^{\prime}, \mathcal O_{\mathcal X^{\prime}})/\mathfrak p_i$
are of the same rank as the $\mathcal O$-modules.
This is the case since they are isomorphic $\mathcal O$-modules.
Indeed, the composite map
$$\mathcal O[t] \xrightarrow{\bar {} \circ \sigma^*} \Gamma(\mathcal X^{\prime}, \mathcal O_{\mathcal X^{\prime}})/(f^{\prime}) \to
\Gamma(\mathcal X^{\prime}, \mathcal O_{\mathcal X^{\prime}})/\mathfrak p_i$$
is an $\mathcal O$-algebra epimorphism and the kernel of this epimorphism is the ideal
$\mathfrak q_i=(h_i)$.

Whence the assertion (f) and whence the Theorem.

\end{proof}

\begin{proof}[Proof of Theorem \ref{ElementaryNisSquareNew}]
We need to find $\sigma$, $h$ and $g$
which enjoy the properties
(a) to (d) from the Theorem.
For that we will use notation from Theorem
\ref{ElementaryNisSquare}.

{\it Take $\sigma$ as in Theorem}
\ref{ElementaryNisSquare}.
Since
$\mathcal X^{\prime}$
is a regular affine irreducible and
$\sigma: \mathcal X^{\prime} \to \Aff^1_U$
is finite surjective
the induced $\mathcal O$-algebra homomorphism
$\sigma^*: \mathcal O[t] \to \Gamma(\mathcal X^{\prime}, \mathcal O_{\mathcal X^{\prime}})$
is a monomorphism. We will regard
below the $\mathcal O$-algebra
$\mathcal O[t]$
as a subalgebra via $\sigma^*$.

{\it Take $h \in \mathcal O[t]$ as in the item
\emph{(f)} of Theorem}
\ref{ElementaryNisSquare}.

Let
$I(\mathcal Z^{\prime\prime})\subseteq \Gamma(\mathcal X^{\prime}, \mathcal O_{\mathcal X^{\prime}})$
be the ideal defining the closed
subscheme
$\mathcal Z^{\prime\prime}$ of the scheme
$\mathcal X^{\prime}$.
Using the items (b) and (c) of Theorem
\ref{ElementaryNisSquare}
find an element
$g \in I(\mathcal Z^{\prime\prime})$
such that \\
(1) $(f^{\prime})+(g)=\Gamma(\mathcal X^{\prime}, \mathcal O_{\mathcal X^{\prime}})$, \\
(2) $Ker((\Delta^{\prime})^*)+(g)=\Gamma(\mathcal X^{\prime}, \mathcal O_{\mathcal X^{\prime}})$, \\
(3) $\sigma_g=\sigma|_{\mathcal X^{\prime}_g}: \mathcal X^{\prime}_g \to \Aff^1_U$ is \'{e}tale.\\
With this choice of $\sigma$, $h$ and $g$ complete the proof of Theorem
\ref{ElementaryNisSquareNew}. The assertions (a) and (c) hold by our choice of $g$.
The assertion (d) holds, since
$\sigma^*(h) \in (f^{\prime})$.
It remains to prove the assertion (b).
The morphism $\sigma$ is finite. Hence the
$\mathcal O[t]$-algebra
$\Gamma(\mathcal X^{\prime}, \mathcal O_{\mathcal X^{\prime}})_g$
is finitely generated. The scheme
$\mathcal X^{\prime}$
is regular and irreducible. Thus,
the ring
$\Gamma(\mathcal X^{\prime}, \mathcal O_{\mathcal X^{\prime}})$
is a domain. The homomorphism
$\sigma^*$
is injective. Hence, the element
$h$ is not zero and is not a zero divisor in
$\Gamma(\mathcal X^{\prime},\mathcal O_{\mathcal X^{\prime}} )_g$.

It remains to check that
$\mathcal O[t]/(h)=\Gamma(\mathcal X^{\prime},\mathcal O_{\mathcal X^{\prime}} )_g/h\Gamma(\mathcal X^{\prime},\mathcal O_{\mathcal X^{\prime}} )_g$.
Firstly, by the choice of $h$ and by the item (a) of Theorem
\ref{ElementaryNisSquare}
one has
$\mathcal O[t]/(h)=\Gamma(\mathcal X^{\prime},\mathcal O_{\mathcal X^{\prime}})/(f^{\prime})$.
Secondly, by the property (1) of the element $g$ one has
$\Gamma(\mathcal X^{\prime},\mathcal O_{\mathcal X^{\prime}})/(f^{\prime})=
\Gamma(\mathcal X^{\prime},\mathcal O_{\mathcal X^{\prime}})_g/f^{\prime}\Gamma(\mathcal X^{\prime},\mathcal O_{\mathcal X^{\prime}})_g$.
Finally, by the items (c) and (a) of Theorem
\ref{ElementaryNisSquare}
one has
\begin{equation}
\label{Prelocalization}
\Gamma(\mathcal X^{\prime},\mathcal O_{\mathcal X^{\prime}})/(f^{\prime}) \
\times \Gamma(\mathcal X^{\prime},\mathcal O_{\mathcal X^{\prime}})/I(\mathcal Z^{\prime\prime})=
\Gamma(\mathcal X^{\prime},\mathcal O_{\mathcal X^{\prime}})/(h).
\end{equation}
Localizing both sides of (\ref{Prelocalization}) in $g$ one gets an equality
$$\Gamma(\mathcal X^{\prime},\mathcal O_{\mathcal X^{\prime}})_g/f^{\prime}\Gamma(\mathcal X^{\prime},\mathcal O_{\mathcal X^{\prime}})_g=
\Gamma(\mathcal X^{\prime},\mathcal O_{\mathcal X^{\prime}})_g/h\Gamma(\mathcal X^{\prime},\mathcal O_{\mathcal X^{\prime}})_g,$$
hence
$$\mathcal O[t]/(h)=
\Gamma(\mathcal X^{\prime},\mathcal O_{\mathcal X^{\prime}})/(f^{\prime})=
\Gamma(\mathcal X^{\prime},\mathcal O_{\mathcal X^{\prime}})_g/f^{\prime}\Gamma(\mathcal X^{\prime},\mathcal O_{\mathcal X^{\prime}})_g=
\Gamma(\mathcal X^{\prime},\mathcal O_{\mathcal X^{\prime}})_g/h\Gamma(\mathcal X^{\prime},\mathcal O_{\mathcal X^{\prime}})_g.
$$
Whence the Theorem.

\end{proof}

\section{Proof of Theorem \ref{ThEquatingGroups}}
\label{ProofThEquatingGroups}
\begin{prop}
\label{PropEquatingGroups}
Let $S$ be a regular semi-local
irreducible scheme.
Assume that all the closed points of $S$ have {\bf finite residue fields}.
Let $G_1,G_2$ be two
{\bf semi-simple}
$S$-group schemes which are twisted forms of each other. Further, let $T\subset S$ be a
closed sub-scheme of $S$ and $\varphi:G_1|_T\to G_2|_T$ be an
$S$-group scheme isomorphism. Then there exists a finite \'{e}tale
morphism $S^{\prime}\xra{\pi}S$ together with its section
$\delta:T\to S^{\prime}$ over $T$ and an $S^{\prime}$-group scheme
isomorphism
$\Phi:\pi^*{G_1}\to\pi^*{G_2}$
such that
$\delta^*(\Phi)=\varphi$.
\end{prop}

\begin{proof}
The proof literally repeats the proof of
\cite[Prop.5.1]{PSV}
except exactly one reference,
which is the reference to
\cite[Lemma 7.2]{OP2}.
That reference one has to replace with the reference to the following
\begin{lem}
\label{OjPan}
Let $S=\text{Spec}(R)$ be a regular
semi-local scheme such that
{\bf the residue field at any of its closed point is finite}.
Let $T$ be a closed
subscheme of $S$. Let $\bar X$ be a closed subscheme of
$\Bbb P^d_S=\text{Proj}(S[X_0,\dots,X_d])$ and
$X=\bar X\cap \Aff^d_S$, where $\Bbb A^d_S$ is the affine space defined by
$X_0\neq0$. Let
$X_{\infty}=\bar X\setminus X$ be the intersection of $\bar X$ with the
hyperplane at infinity
$X_0=0$. Assume that over $T$ there exists a section
$\delta:T\to X$
of the canonical projection $X\to S$. Assume further that \\
\smallskip
{\rm{(1)}} $X$ is smooth and equidimensional over $S$,
of relative dimension $r$;\\
{\rm{(2)}} For every closed
point $s\in S$ the closed fibres of $X_\infty$ and $X$
satisfy
$$\text{dim}(X_\infty(s))<\text{dim}(X(s))=r\;.$$
\smallskip
Then there exists a closed subscheme $\tilde S$ of $X$ which is finite
\'etale over
$S$ and contains $\delta(T)$.
\end{lem}

{\it The proof of the lemma is given below} and repeats literally the proof of
\cite[Lemma 7.2]{OP2}. The only difference is that we refer below
to a Poonen's article
\cite{Poo1}
on Bertini theorems over finite fields rather than to Artin's result.

Since $S$ is semilocal, after a linear change
of coordinates we may assume that $\delta$ maps $T$ into
the closed subscheme of $\Pro^d_T$ defined by
$X_1=\dots=X_d=0$. For each closed fibre $\Pro^d_s$ of $\Pro^d_S$
using repeatedly
\cite[Thm.1.2]{Poo1},
we can choose a family of {\bf homogeneous} polynomials $H_1(s),\dots,H_r(s)$
(in general of increasing degrees)
such that the subscheme $Y(s)$ of $\Pro^d_S(s)$ defined by the equations
$$H_1(s)=0\,,\,\dots\,,\, H_r(s)=0$$
intersects $X(s)$ transversally, contains the point
$[1:0:\dots :0]$
and
avoids $X_{\infty}(s)$.
  By the chinese
remainders' theorem there exists a common  lift
$H_i\in R[X_0,\dots,X_d]$ of all polynomials $H_i(s)$, $s\in
\text{Max}(R)$. We
may choose this common lift $H_i$ such that $H_i(1,0,\dots,0)=0$. Let $Y$ be
the closed subscheme of $\Pro^d_S$ defined by
$$H_1=0\,,\dots,H_r=0\;.$$

{\it We claim that the subscheme  $\tilde S=Y\cap X$  has the
required properties}. Note first that $X\cap Y$ is finite
over $S$.
In fact, $X\cap Y=\bar X\cap Y$, which
is projective over $S$ and such that every closed
fibre (hence every fibre) is finite. Since the closed
fibres of $X\cap Y$ are finite \'etale over the closed points
of $S$, to show that $X\cap Y$ is finite \'etale over $S$ it
only remains to show that it is flat over $S$. Noting that
$X\cap Y$ is defined in every closed fibre by a regular
sequence of equations and localizing at each closed point of
$S$, we see that flatness follows from
\cite[Lemma 7.3]{OP2}.

\end{proof}

Let $k$ be {\bf a finite field}.
Let $U$ be as in Definition
\ref{DefnNiceTriple}.
Let $S^{\prime}$ be an irreducible regular semi-local scheme over $k$ and
$p: S^{\prime} \to U$ be a $k$-morphism.
Let $T^{\prime}\subset S^{\prime}$ be a closed sub-scheme of $S^{\prime}$
such that the restriction
$p|_{T^{\prime}}: T^{\prime} \to U$ is an isomorphism.
We will assume below that $dim(T^{\prime}) < dim(S^{\prime})$,
where $dim$ is the Krull dimension.
For any closed point $u \in U$ and any $U$-scheme $V$ let $V_u=u\times_U V$ be
the fibre of the scheme $V$ over the point $u$.
For a finite set $A$ denote $\sharp A$ the cardinality of $A$.

\begin{lem}
\label{SmallAmountOfPoints}
Assume that all the closed points of $S^{\prime}$ have {\bf finite residue fields}.
Then there exists a finite \'{e}tale morphism
$\rho: S^{\prime\prime} \to S^{\prime}$ (with an irreducible scheme $S^{\prime\prime}$)
and a section
$\delta^{\prime}: T^{\prime} \to S^{\prime\prime}$
of $\rho$ over $T^{\prime}$ such that the following holds
\begin{itemize}
\item[\rm{(1)}]
for any closed point $u \in U$ let $u^{\prime} \in T^{\prime}$ be a unique point such that
$p(u^{\prime})=u$, then the point
$\delta^{\prime}(u^{\prime}) \in S^{\prime\prime}_u$ is the only
$k(u)$-rational point of $S^{\prime\prime}_u$,
\item[\rm{(2)}]
for any closed point $u \in U$ and any integer $r \geq 1$
one has
$$\sharp\{z \in S^{\prime\prime}_u| \text{deg}[z:u]=r \} \leq \ \sharp\{x \in \Aff^1_u| \text{deg}[z:u]=r \}$$
\end{itemize}
\end{lem}

\begin{proof}[Proof of Theorem \ref{ThEquatingGroups}]
We can start by
almost literally repeating arguments from the proof of \cite[Lemma
8.1]{OP1}, which involve the following purely geometric lemma
\cite[Lemma 8.2]{OP1}.
\par
For reader's convenience below we state that Lemma adapting
notation to the ones of Section \ref{NiceTriples}.
Namely, let $U$ be as in Definition \ref{DefnNiceTriple} and let
$(\mathcal X,f,\Delta)$ be a nice triple over $U$. Further, let
$G_{\mathcal X}$ be a simple simply-connected $\mathcal X$-group
scheme, $G_U:=\Delta^*(G_{\mathcal X})$, and let $G_{\const}$ be
the pull-back of $G_U$ to $\mathcal X$. Finally, by the definition
of a nice triple there exists a finite surjective morphism
$\Pi:\mathcal X\to\Aff^1\times U$ of $U$-schemes.
\begin{lem}
\label{Lemma_8_2} Let $\mathcal Y$ be a closed nonempty sub-scheme
of $\mathcal X$, finite over $U$. Let $\mathcal V$ be an open
subset of $\mathcal X$ containing $\Pi^{-1}(\Pi(\mathcal Y))$. There
exists an open set $\mathcal W \subseteq \mathcal V$ still
containing $q_U^{-1}(q_U(\mathcal Y))$ and endowed with a finite
surjective morphism
$\Pi^*: \mathcal W\to\Aff^1\times U$ {\rm(}in general
$\neq\Pi${\rm)}.
\end{lem}
Let $\Pi:\mathcal X\to\Aff^1\times U$ be the above finite
surjective $U$-morphism.
The following diagram summarises the situation:
$$ \xymatrix{
{}&\mathcal Z \ar[d]&{}\\
{\mathcal X - \mathcal Z\ }\ar@{^{(}->}@<-2pt>[r]&\mathcal X\ar@<2pt>[d]^{q_U}\ar[r]^>>>>{\Pi}& \Aff^1 \times U\\
{}& U\ar@<2pt>[u]^{\Delta}&{} } $$
\noindent
Here $\mathcal Z$ is the closed sub-scheme defined by the equation
$f=0$. By assumption, $\mathcal Z$ is finite over $U$. Let
$\mathcal Y=\Pi^{-1}(\Pi(\mathcal Z \cup \Delta(U)))$. Since
$\mathcal Z$ and $\Delta(U)$ are both finite over $U$ and since
$\Pi$ is a finite morphism of $U$-schemes, $\mathcal Y$ is also
finite over $U$. Denote by $y_1,\dots,y_m$ its closed points and
let $S=\text{Spec}(\mathcal O_{\mathcal X,y_1,\dots,y_m})$. Set
$T=\Delta(U)\subseteq S$. Further, let $G_U=\Delta^*(G_{\mathcal
X})$ be as in the hypotheses of Theorem
\ref{ThEquatingGroups} and
let
$G_{\const}$
be the pull-back of
$G_U$ to $\mathcal X$. Finally,
let
$\varphi:G_{\const}|_T \to G_{\mathcal X}|_T$
be the canonical
isomorphism. Recall that by assumption $\mathcal X$ is $U$-smooth and irreducible,
and thus $S$ is regular and irreducible.
\par
By Proposition
\ref{PropEquatingGroups} there exists a finite
\'etale covering $\theta_0:S^{\prime}\to S$, a section
$\delta:T\to S^{\prime}$ of $\theta_0$ over $T$ and an isomorphism
$$ \Phi_0:\theta^*_0(G_{\const,S})\to\theta^*_0(G_{\mathcal X}|_S) $$
\noindent
such that $\delta^*\Phi_0=\varphi$.
{\bf Replacing $S^{\prime}$ with a connected component of $S^{\prime}$ which contains
$T^{\prime}:=\delta(T)=\delta(\Delta(U))$
we may and will assume that $S^{\prime}$ is irreducible.}

Let $p=q_U \circ \theta_0: S^{\prime} \to U$.
By Lemma \ref{SmallAmountOfPoints}
there exists a finite \'{e}tale morphism
$\rho: S^{\prime\prime} \to S^{\prime}$ (with an irreducible scheme $S^{\prime\prime}$)
and a section
$\delta^{\prime}: T^{\prime} \to S^{\prime\prime}$
of $\rho$ over $T^{\prime}$ such that the properties (1) and (2) from
Lemma \ref{SmallAmountOfPoints} holds. Set
$\delta^{\prime\prime}=\delta^{\prime} \circ \delta: T \to S^{\prime\prime}$
and
$\theta^{\prime\prime}_0=\theta_0 \circ \rho: S^{\prime\prime} \to U$.
We are also given the
$S^{\prime\prime}$-group scheme isomorphism
$$\rho^*(\Phi_0): (\theta^{\prime\prime}_0)^*(G_{\const,S}) \to (\theta^{\prime\prime}_0)^*(G_{\mathcal X}|_S)$$

We can extend these data to a
neighborhood $\mathcal V$ of $\{y_1,\dots,y_n\}$ and get the
diagram
\begin{equation}
\xymatrix{
     {}  &  S^{\prime\prime} \ar[d]^{\theta^{\prime\prime}_0} \ar@{^{(}->}@<-2pt>[r]  & {\mathcal V}^{\prime\prime}  \ar[d]_{\theta} &\\
     T \ar@{^{(}->}@<-2pt>[r] \ar[ur]^{
\delta^{\prime\prime}} & S \ar@{^{(}->}@<-2pt>[r]  &   \mathcal V \ar@{^{(}->}@<-2pt>[r]  &  \mathcal X &\\
    }
\end{equation}
\noindent
where
$\theta: {\mathcal V}^{\prime\prime}\to\mathcal V$
finite \'etale, and an
isomorphism
$\Phi:\theta^*(G_{\const})\to\theta^*(G_{\mathcal X})$.
\par
Since $T$ isomorphically projects onto $U$, it is still closed
viewed as a sub-scheme of $\mathcal V$. Note that since $\mathcal
Y$ is semi-local and $\mathcal V$ contains all of its closed
points, $\mathcal V$ contains $\Pi^{-1}(\Pi(\mathcal Y))=\mathcal
Y$. By Lemma \ref{Lemma_8_2} there exists an open subset $\mathcal
W\subseteq\mathcal V$ containing $\mathcal Y$ and endowed with a
finite surjective $U$-morphism $\Pi^*:\mathcal W\to\Aff^1\times
U$.
\par
Let $\mathcal X^{\prime}=\theta^{-1}(\mathcal W)$,
$f^{\prime}=\theta^{*}(f)$, $q^{\prime}_U=q_U\circ\theta$, and let
$\Delta^{\prime}:U\to\mathcal X^{\prime}$ be the section of
$q^{\prime}_U$ obtained as the composition of $\delta^{\prime\prime}$ with
$\Delta$. We claim that the triple
$(\mathcal X^{\prime},f^{\prime},\Delta^{\prime})$ is a nice triple over $U$. Let us
verify this. Firstly, the structure morphism
$q^{\prime}_U:\mathcal X^{\prime} \to U$ coincides with the
composition
$$
\mathcal X^{\prime}\xra{\theta}
\mathcal W\hra\mathcal X\xra{q_U} U.
$$
\noindent
Thus, it is smooth. The element $f^{\prime}$ belongs to the ring
$\Gamma(\mathcal X^{\prime},\mathcal O_{\mathcal X^{\prime}})$,
the morphism $\Delta^{\prime}$ is a section of $q^{\prime}_U$.
Each component of each fibre of the morphism $q_U$ has dimension
one, the morphism $\mathcal X^{\prime}\xra{\theta}\mathcal
W\hra\mathcal X$ is \'{e}tale. Thus, each component of each fibre
of the morphism $q^{\prime}_U$ is also of dimension one. Since
$\{f=0\} \subset {\mathcal W}$ and $\theta: \mathcal X^{\prime}
\to {\mathcal W}$ is finite, $\{f^{\prime}=0\}$ is finite over
$\{f=0\}$ and hence  also over $U$. In other words, the $\mathcal
O$-module $\Gamma(\mathcal X^{\prime},\mathcal O_{\mathcal
X^{\prime}})/f^{\prime}\cdot\Gamma(\mathcal X^{\prime},\mathcal
O_{\mathcal X^{\prime}})$ is finite. The morphism $\theta:
\mathcal X^{\prime}\to\mathcal W$ is finite and surjective.
We have constructed above in
Lemma \ref{Lemma_8_2}
the finite surjective morphism
$\Pi^*:\mathcal W\to\Aff^1\times U$. It follows that
$\Pi^*\circ\theta:\mathcal X^{\prime}\to\Aff^1\times U$ is finite
and surjective.
\par
Clearly, the \'{e}tale morphism
$\theta:\mathcal X^{\prime} \to\mathcal X$
is a morphism between the nice triples, with $h^{\prime}=1$.
\par
Denote the restriction of $\Phi$ to $\mathcal X^{\prime}$ simply
by $\Phi$. The equality $(\Delta^{\prime})^*{\Phi}=\id_{G_U}$
holds by the very construction of the isomorphism $\Phi$. Theorem
follows.

All the closed points of the sub-scheme
$\{f=0\} \subset \mathcal X$
are in
$S$.
The morphism
$\theta$ is finite and
$\theta^{-1}(S)=S^{\prime\prime}$.
Thus all the closed points of the sub-scheme
$\{f^{\prime}=0\} \subset \mathcal X^{\prime}$
are in
$S^{\prime\prime}$.
Now the properties
(1) and (2)
of the
$U$-scheme
$S^{\prime\prime}$
show that
the assertions (2) and (3) of Theorem
\ref{ThEquatingGroups}
do hold
for the closed sub-scheme $\mathcal Z^{\prime}$ of $\mathcal X^{\prime}$ defined by $\{f^{\prime}=0\}$.

Theorem
\ref{ThEquatingGroups}
follows.

\end{proof}

\section{A basic nice triple}
\label{BasicTriple}
Let $k$ be {\bf a finite field}. Fix a smooth geometrically irreducible affine $k$-scheme $X$,
and a finite family of {\bf closed} points $x_1,x_2, \dots , x_n$ on $X$, and a
non-zero function $\ttf\in k[X]$, which vanishes at each of $x_i$'s for $i=1,2,\dots,n$.
Let $\mathcal O=\mathcal O_{X,\{x_1,x_2,\dots,x_n\}}$ be the semi-local ring of the family $x_1,x_2, \dots , x_n$ on $X$,
$U=Spec(\mathcal O)$ and
$\can:U\hra X$ the canonical inclusion of schemes.
The definition of a nice triple over $U$ is given in 3.1.
The main aim of the present section is to prove the following

\begin{prop}
\label{BasicTripleProp}
One can shrink $X$ such that $x_1,x_2, \dots , x_n$ are still in $X$ and $X$ is affine, and then to construct a nice triple
$(q_U: \mathcal X \to U, \Delta, f)$ over $U$ and an essentially smooth morphism $q_X: \mathcal X \to X$ such that
$q_X \circ \Delta= can$, $f=q^*_X(\text{f})$ and the set of closed points of $\Delta(U)$ is
contained in the set of closed points of $\{f=0\}$.
\end{prop}

\begin{proof}
By Proposition \ref{ArtinsNeighbor} there exist a Zariski open
neighborhood $X^0$ of the family $\{x_1,x_2,\dots,x_n\}$ and an
almost elementary fibration $p:X^0\to S$, where $S$ is an open sub-scheme
of the projective space $\Pro^{\mydim X-1}$, such that
$$ p|_{\{\ttf=0\}\cap X^0}:\{\ttf=0\}\cap X^0\to S $$
\noindent
is finite surjective. Let $s_i=p(x_i)\in S$, for each $1\le i\le
n$. Shrinking $S$, we may assume that $S$ is {\it affine \/} and
still contains the family $\{s_1,s_2,\dots,s_n\}$. Clearly, in
this case $p^{-1}(S)\subseteq X^0$ contains the family
$\{x_1,x_2,\dots,x_n\}$. We replace $X$ by $p^{-1}(S)$ and $\ttf$
by its restriction to this new $X$.
\par
In this way we get an almost elementary fibration $p:X\to S$ such that
$$ \{x_1,\dots,x_n\}\subset\{\ttf=0\}\subset X, $$
\noindent
$S$ is an open affine sub-scheme in the projective space
$\Pro^{\mydim X-1}$, and the restriction
$p|_{\{\ttf=0\}}:\{\ttf=0\}\to S$ of $p$ to the vanishing locus of $\ttf$
is a finite surjective morphism. In other words, $k[X]/(\ttf)$ is
finite as a $k[S]$-module.
\par
As an open affine sub-scheme of the projective space $\Pro^{\mydim
X-1}$ the scheme $S$ is regular. By Proposition~\ref{CartesianDiagram}
one can shrink $S$ in such a way that $S$
is still affine, contains the family $\{s_1,s_2,\dots,s_n\}$ and
there exists a finite surjective morphism
$$ \pi:X\to\Aff^1\times S $$
\noindent
such that $p=\pr_S\circ\pi$. Clearly, in this case
$p^{-1}(S)\subseteq X$ contains the family
$\{x_1,x_2,\dots,x_n\}$. We replace $X$ by $p^{-1}(S)$ and $\ttf$
by its restriction to this new $X$.
\par
In this way we get an almost elementary fibration $p:X\to S$ such that
$$ \{x_1,\dots,x_n\}\subset\{\ttf=0\}\subset X, $$
\noindent
$S$ is an open affine sub-scheme in the projective space
$\Pro^{\mydim X-1}$, and the restriction
$p|_{\{\ttf=0\}}:\{\ttf=0\}\to S$
is a finite surjective morphism. Eventually we conclude that there
exists a finite surjective morphism $\pi:X\to\Aff^1\times S $ such that $p=\pr_S\circ\pi$.
\par
Let $p_U=p\circ\can:U\to S$, {\bf where $U=Spec(\mathcal O)$ and $can: U \hra X$ are as above}. Further, we consider the fibre
product
$$
\mathcal X:=U\times_{S}X.
$$
 Then the canonical projections
$q_U:\mathcal X\to U$ and $q_X:\mathcal X\to X$ and the diagonal
morphism $\Delta:U\to\mathcal X$ can be included in the following
diagram
\begin{equation}
\label{SquareDiagram}
    \xymatrix{
     \mathcal X\ar[d]_{q_U}\ar[rr]^{q_X} &&  X   & \\
     U \ar[urr]_{\can} \ar@/_0.8pc/[u]_{\Delta} &\\
    }
\end{equation}
where
\begin{equation}
\label{DeltaQx} q_X\circ\Delta=\can
\end{equation}
and
\begin{equation}
\label{DeltaQu} q_U\circ\Delta=\id_U.
\end{equation}
Note that $q_U$ is {\it a smooth morphism with geometrically
irreducible fibres of dimension one}. Indeed, observe that $q_U$
is a base change via $p_U$ of the morphism $p$ which has the
desired properties.
Note that $\mathcal X$ is irreducible. Indeed,
$U$ is irreducible and the fibre of $q_U$ over the generic point of $U$ is irreducible.

Taking the base change via $p_U$ of the finite
surjective morphism $\pi:X\to\Aff^1\times S$, we get {\it a finite
surjective morphism
$$ \Pi:\mathcal X\to\Aff^1\times U $$
\noindent
such that\/} $q_U=\pr_U\circ\Pi$, where $pr_U:\Aff^1\times U\to U$ is the natural projection.

Set $f:=q_X^*(\ttf)$. The $\mathcal O_{X,\{x_1,x_2,\dots,x_n\}}$-module
$\Gamma(\mathcal X,\mathcal O_{\mathcal X})/f\cdot\Gamma(\mathcal X,\mathcal O_{\mathcal X})$
is finite, since the $k[S]$-module $k[X]/\ttf\cdot k[X]$ is finite.

Now the data
\begin{equation}
\label{FirstNiceTriple} (q_U:\mathcal X\to U,f,\Delta)
\end{equation}
form an example of a {\it nice triple\/} as in
Definition
\ref{DefnNiceTriple}. Moreover, we have

\begin{clm}
\label{DeltaIsWellDefined} The schemes $\Delta(U)$ and $\{f=0\}$
are both semi-local and the set of closed points of $\Delta(U)$ is
contained in the set of closed points of $\{f=0\}$.
\end{clm}
This holds since the set $\{x_1,x_2,\dots,x_n\}$ is contained in
the vanishing locus of the function $\ttf$.
{\bf The nice triple
(\ref{FirstNiceTriple})
together with the essentially smooth morphism $q_X$ are the required one.
Whence the proposition.}
\end{proof}

\section{Proof of Theorem \ref{MainHomotopy}}
\label{MainConstruction}
The main result of this Section is Corollary
\ref{Ptandht}
(= Theorem
\ref{MainHomotopy}).

Fix a $k$-smooth irreducible affine $k$-scheme $X$, a finite family of points
$x_1,x_2,\dots,x_n$ on $X$, and set $\mathcal O:=\mathcal
O_{X,\{x_1,x_2,\dots,x_n\}}$ and $U:= \text{Spec}(\mathcal O)$.
Further, consider a simple simply connected $U$-group scheme $G$
and a principal $G$-bundle $P$ over $\mathcal O$ which is trivial
over $K$ for the field of fractions $K$ of $\mathcal O$.
We may and will
assume that for certain $0 \neq \text{f} \in \mathcal O$ the principal
$G$-bundle $P$ is trivial over $\mathcal O_{\text{f}}$.

Shrinking
$X$ if necessary, we may secure the following properties
\par\smallskip
(i) The points $x_1,x_2,\dots,x_n$ are still in $X$ and $X$ is affine.
\par\smallskip
(ii) The group scheme $G$ is defined over $X$ and it is a simple
group scheme. We will often denote this $X$-group scheme by $G_X$
and write $G_U$ for the original $G$.
\par\smallskip
(iii)
The principal $G_U$-bundle $P$ is the restriction to $U$ of a principal $G_X$-bundle $P_X$ over $X$
and
$\ttf\in k[X]$.
We often will write $P_U$ for the original principal $G_U$-bundle $P$ over $U$.
\par\smallskip
(iv) The restriction $P_{\text{f}}$ of the bundle $P_X$ to the
principal open subset $X_{\text{f}}$ is trivial and $\text{f}$
vanishes at each $x_i$'s.

{\bf If we shrink $X$ further such that the property (i) is secured, then we automatically secure the properties
(ii) to (iv). For any such $X$ we will write $can: U \hra X$ for the canonical embedding.}

After substituting $k$ by its algebraic closure $\tilde k$ in $k[X]$,
we can assume that $X$ is a $\tilde k$-smooth geometrically irreducible affine $\tilde k$-scheme.
To simplify the notation, we will
continue to denote this new $\tilde k$ by $k$.

\smallskip
In particular, we are given now the smooth geometrically irreducible affine
$k$-scheme $X$, the finite family of points $x_1,x_2,\dots,x_n$ on
$X$, and the non-zero function $\ttf\in k[X]$ vanishing at each
point $x_i$.
{\bf
We may shrink $X$
further securing the property (i)
and
construct
the nice triple
$(q_U: \mathcal X \to U, \Delta, f)$
over $U$
and the essentially smooth morphism
$q_X: \mathcal X \to U$ as in
Proposition
\ref{BasicTripleProp}.}
{\bf Since the property (i) is secured the properties (ii) to (iv) are secured too.}
{\bf Consider the $\mathcal X$-group scheme
$G_{\mathcal X}:=(q_X)^*(G_X)$.
Note that the $U$-group scheme
$\Delta^*(G_{\mathcal X})$
coincides with $G_U$ from the item (ii) since $can=q_X \circ \Delta$ by
Proposition
\ref{BasicTripleProp}. Consider one more $\mathcal X$-group scheme, namely }
$$G_{const}:=(q_U)^*(\Delta^*(G_{\mathcal X}))=(q_U)^*(G_U).$$

By Theorem \ref{ThEquatingGroups}
there exists a morphism of nice triples
$$ \theta: (q^{\prime}_U: \mathcal X^{\prime} \to U,f^{\prime},\Delta^{\prime})\to (q_U: \mathcal X \to U,f,\Delta) $$
\noindent
and an isomorphism
\begin{equation}
\label{KeyEquation} \Phi: \theta^*(G_{\const}) \to
\theta^*(G_{\mathcal X})=: G_{\mathcal X^{\prime}}
\end{equation}
of $\mathcal X^{\prime}$-group schemes such that
$(\Delta^{\prime})^*(\Phi)=\id_{G_U}$ and such that
the closed sub-scheme
$\mathcal Z^{\prime}:=\{f^{\prime}=0\}$ satisfies the conditions
{\rm{(2)}} and {\rm{(3)}} from
Theorem \ref{ThEquatingGroups}.
Set
\begin{equation}
\label{QprimeX} q^{\prime}_X=q_X\circ\theta:\mathcal X^{\prime}\to X.
\end{equation}
Recall that
\begin{equation}
\label{QprimeU} q^{\prime}_U=q_U\circ\theta:\mathcal X^{\prime}\to
U,
\end{equation}
since
$\theta$
is a morphism of nice triples.\\

Note that, since by Claim~\ref{DeltaIsWellDefined} $f$ vanishes on all closed points of $\Delta(U)$,
and $\theta$ is a morphism of nice triples, $f^{\prime}$ vanishes on all closed points of $\Delta^{\prime}(U)$ as well.
Therefore, the nice triple
$(q^{\prime}_U: \mathcal X^{\prime} \to U, f^{\prime}, \Delta^{\prime}: U \to \mathcal X^{\prime})$
is subject to Theorem~\ref{ElementaryNisSquareNew}.

By Theorem \ref{ElementaryNisSquareNew} there exists a finite
surjective morphism $\sigma:\mathcal X^{\prime}\to\Aff^1\times U$
of $U$-schemes,
a monic polinomial $h \in ker(\sigma^*)$
and an element $g \in \Gamma(\mathcal X^{\prime},\mathcal O_{\mathcal X^{\prime}} )$
which enjoys the properties
\rm{(a)} to \rm{(d)} from that Theorem.
The item \rm{(b)} of Theorem \ref{ElementaryNisSquareNew}  shows that the cartesian square
\begin{equation}
\label{SquareDiagram2}
    \xymatrix{
     \mathcal X^{\prime}_{gh}  \ar[rr]^{\inc} \ar[d]_{\sigma_{gh}} &&
     \mathcal X^{\prime}_g \ar[d]^{\sigma_g}  &\\
     (\Aff^1 \times U)_{h} \ar[rr]^{\inc} && \Aff^1 \times U &\\
    }
\end{equation}
can be used to glue principal $G$-bundles.

Below, we use this to construct principal
$G_U$-bundles over
$\Aff^1\times U$
out of the following initial data: a principal $G_U$-bundle over
$\mathcal X^{\prime}_{g}$,
the trivial principal $G_U$-bundle over
$(\Aff^1\times U)_{h}$,
and a principal $G_U$-bundle isomorphism of their pull-backs to
$\mathcal X^{\prime}_{gh}$.

Consider $(q^{\prime}_X)^*(P_X)$ as a principal
$(q^{\prime}_U)^*(G_U)=\theta^*(G_{\const})$-bundle via the isomorphism
$\Phi$. Recall that $P_X$ is trivial as a principal $G_X$-bundle over
$X_{\text{f}}$. Therefore,
$(q^{\prime}_X)^*(P_X)$ is trivial as a
principal
$\theta^*(G_{\mathcal X})$-bundle
over
$\mathcal X^{\prime}_{f^{\prime}}$.
So, $(q^{\prime}_X)^*(P_X)$ is trivial over
$\mathcal X^{\prime}_{f^{\prime}}$, when regarded as a principal
$(q^{\prime}_U)^*(G_U)= \theta^*(G_{{\const}})$-bundle via the isomorphism $\Phi$.

Thus, regarded as a principal $G_U$-bundle,
the bundle
$(q^{\prime}_X)^*(P_X)$
over
$\mathcal X^{\prime}$
becomes trivial over
$\mathcal X^{\prime}_{f^{\prime}}$,
and a
fortiori over
$(\mathcal X^{\prime})_{gh}$.
Indeed,
$(\mathcal X^{\prime})_{gh} \subseteq (\mathcal X^{\prime})_{gf^{\prime}}$
by the item
\rm(d) of
Theorem
\ref{ElementaryNisSquareNew}.
Take the trivial $G_U$-bundle over
$(\Aff^1\times U)_{h}$ and an isomorphism
\begin{equation}
\label{Skleika}
\psi: G_U \times_U \mathcal X^{\prime}_{gh}  \to
(q^{\prime}_X)^*(P_X)|_{\mathcal X^{\prime}_{gh} }
\end{equation}
of the principal $G_U$-bundles.
By item \rm{(2)} of Theorem
\ref{ElementaryNisSquareNew}
the triple
$$ (\mathcal O[t],\sigma^*_g: \mathcal O[t] \to \Gamma(\mathcal X^{\prime},\mathcal O_{\mathcal X^{\prime}} )_g, h ) $$
satisfies the hypotheses of
\cite[Prop.2.6.(iv)]{C-TO}.
{\bf
The latter statement implies that
one can find} {\it a principal $G_U$-bundle $\mathcal G_t$
over} $\Aff^1\times U$ such that
\begin{itemize}
\item[(1)]
$\mathcal G_t|_{[(\Aff^1\times U)_{h}]}=G_U \times_U [(\Aff^1\times U)_{h}]$,
\item[(2)]
there is an isomorphism
$\varphi: \sigma_{g}^*(\mathcal G_t) \to (q^{\prime}_X)^*(P_X)|_{\mathcal X^{\prime}_{g}}$
of the principal $G_U$-bundles,
where $(q^{\prime}_X)^*(P_X)$
is regarded as a principal $G_U$-bundle via the $\mathcal X^{\prime}$-group scheme isomorphism $\Phi$
from
(\ref{KeyEquation}).
\end{itemize}


Finally, form the following diagram
\begin{equation}
\label{DeformationDiagram2}
    \xymatrix{
\Aff^1 \times U\ar[drr]_{\pr_U}&&\mathcal
X^{\prime}_{g} \ar[d]^{}
\ar[ll]_{\sigma_{g}}\ar[d]_{q^{\prime}_U}
\ar[rr]^{q^{\prime}_X}&&X &\\
&&U \ar[urr]_{\can}\ar@/_0.8pc/[u]_{\Delta^{\prime}} &\\
    }
\end{equation}
This diagram is well-defined, since by Item \rm(c) of Theorem~\ref{ElementaryNisSquareNew}
the image of the morphism
$\Delta^{\prime}$ lands in $\mathcal X^{\prime}_{g^{\prime}}$.


\begin{thm}
\label{MainData}
The principal $G_{U}$-bundle $\mathcal G_t$ over $\Aff^1 \times U$,
the monic polynomial $h \in \mathcal O[t]$, the diagram~{\rm(\ref{DeformationDiagram2})},
and the isomorphism $\Phi$ from~{\rm(\ref{KeyEquation})} constructed above,
satisfy the following conditions {\rm(1*)}--{\rm(6*)}.
\par\smallskip
(1*) $q^{\prime}_U=\pr_U\circ\sigma_{g}$,
\par\smallskip
(2*) $\sigma_{g}$ is \'etale,
\par\smallskip
(3*) $q^{\prime}_U\circ\Delta^{\prime}=\id_U$,
\par\smallskip
(4*) $q^{\prime}_X\circ\Delta^{\prime}=\can$,
\par\smallskip
(5*) the restriction of $\mathcal G_t$ to $(\Aff^1 \times U)_{h}$ is a trivial
$G_U$-bundle,
\par\smallskip
(6*) $\sigma_{g}^*(\mathcal G_t)$ and $(q^{\prime}_X)^*(P_X)$ are isomorphic as
$G_U$-bundles over
$\mathcal X^{\prime}_{g}$.
Here $(q^{\prime}_X)^*(P_X)$ is regarded as a principal
$G_U$-bundle via the group scheme isomorphism $\Phi$ from {\rm(\ref{KeyEquation})}.
\par\smallskip
\end{thm}

\begin{proof} By the choice of $\sigma$ it is an $U$-scheme
morphism, which proves (1*). By the choice of
$\mathcal X^{\prime}_g\hra\mathcal X^{\prime}$
in Theorem~\ref{ElementaryNisSquareNew},
the morphism $\sigma$ is \'etale on this sub-scheme, hence one gets (2*).
Property (3*) holds for $\Delta^{\prime}$ since
$(q^{\prime}_X: \mathcal X^{\prime} \to U, f^{\prime},\Delta^{\prime})$
is a nice triple
and, in particular,
$\Delta^{\prime}$ is a section of
$q^{\prime}_U$. Property (4*) can be established as follows:
$$ q^{\prime}_X\circ\Delta^{\prime}=
(q_X\circ\theta)\circ\Delta^{\prime}=q_X\circ\Delta=\can. $$
\noindent
The first equality here holds by the definition of $q^{\prime}_X$,
the second one holds since $\theta$ is a
morphism of nice triples; the third one follows from equality
(\ref{DeltaQx}). Property (5*) is just Property (1) in the above
construction of $\mathcal G_t$.
Property (6*) is precisely Property (2) in the construction of
$\mathcal G_t$.
\end{proof}
The composition
$$
s:= \sigma_{g} \circ \Delta^{\prime}:U\to\Aff^1\times U
$$ is a section of the projection $\pr_U$
by the properties (1*) and (3*). Recall that $G_U$ over $U$ is the original group scheme $G$ introduced in the very
beginning of this Section.
Since $U$ is semi-local, we may assume that $s$ {\it is the zero section of the projection} $\Aff^1_U \to U$.
\begin{cor}[{\bf=Theorem \ref{MainHomotopy}}]
\label{Ptandht}
The principal $G_U$-bundle $\mathcal G_t$ over $\Aff^1_U$ and the monic polynomial
$h \in \mathcal O[t]$
are subject to the following conditions
\par\smallskip
(i) the restriction of $\mathcal G_t$ to $(\Aff^1\times U)_{h}$ is a trivial principal
$G_U$-bundle,
\par\smallskip
(ii) the restriction of $\mathcal G_t$ to $\{0\} \times U$ is the original $G_U$-bundle $P_U$.
\end{cor}

\begin{proof}
The property (i) is just the property (5*) above. Now by (6*) the $G_U$-bundles
$$\mathcal G_t|_{\{0\}\times U}=s^*(\mathcal G_t)=
(\Delta^{\prime})^*(\sigma_{g}^*(\mathcal G_t)) \ \text{and}
 \ {(\Delta^{\prime})}^*(q^{\prime}_X)^*(P_X)=\can^*(P_X)$$
are isomorphic, since ${\Delta^{\prime}}^*(\Phi)=\id_{G_U}$.
It remains to recall that the principal $G_U$-bundle
$\can^*(P_X)$ is the original $G_U$-bundle $P_U$
by the choice of $P_X$. Whence the Corollary.

\end{proof}

%


\end{document}